\documentclass[12pt]{amsart}
\usepackage{ifthen,latexsym,amssymb,amsmath}
\usepackage[T1]{fontenc}
\makeatletter
\@namedef{subjclassname@2010}{%
  \textup{2010} Mathematics Subject Classification}
\makeatother

\newtheorem{theorem}{Theorem}
\newtheorem{lemma}[theorem]{Lemma}
\newtheorem{proposition}[theorem]{Proposition}

\begin{document}

\baselineskip=17pt

\title[cyclotomic polynomials]{Products of cyclotomic polynomials on unit circle}

\author[B. Bzd\k{e}ga]{Bart{\l}omiej Bzd\k{e}ga}
\address{Adam Mickiewicz University \\ Faculty of Mathematics and Computer Sciences \\
61-614 Pozna\'n, Poland}
\email{exul@amu.edu.pl}

\date{}

\begin{abstract}
We present a method to deal with the values of polynomials of type $P(z)=\prod_{d\in D}(1-z^d)^{j_d}$ on the unit circle. We use it to improve the known bounds on various measures of coefficients of cyclotomic and similar polynomials.
\end{abstract}

\subjclass[2010]{11B83, 11C08}

\keywords{cyclotomic polynomial, binary cyclotomic polynomial, ternary cyclotomic polynomial, Beiter conjecture, unit circle, jump one property, relatives of cyclotomic polynomials}

\maketitle

\section{Introduction and earlier results}

Let $\Phi_n(z)=\sum_ma_n(m)z^m$ be the $n$th cyclotomic polynomial, where we assume that $n$ is odd and square free. Then the order of $\Phi_n$ is the number $\omega(n)$ of prime divisors of $n$. Put
\begin{align*}
A_n=\max_m|a_n(m)|, \quad & \quad S_n=\sum_m|a_n(m)|, \\
Q_n=\sum_ma_n(m)^2, \quad & \quad L_n=\max_{|z|=1}|\Phi_n(z)|.
\end{align*}
We have $\deg\Phi_n=\varphi(n)\le n$, so these quantities are bounded by the inequalities
$$L_n/n \le S_n/n \le \sqrt{Q_n/n} \le A_n.$$
The author \cite{Bzdega-Height,Bzdega-GeneralBeiter} proved that for $n=p_1p_2\ldots p_k$, $p_1<p_2<\ldots<p_k$ we have
$$(c+\epsilon_k)^{2^k} \le \sup_{\omega(n)=k}\frac{L_n/n}{M_n} \le \frac{A_n}{M_n} \le (C+\epsilon_k)^{2^k},$$
where $c\approx0.71$, $C\approx0.95$, $M_n=\prod_{j=1}^{k-2}p_j^{2^{k-j-1}-1}$ and $\epsilon_k\to0$ for $k\to\infty$. Moreover $M_n$ gives the optimal order, i.e. it cannot be replaced by any smaller (in lexicographical sense) product of powers of prime factors of $n$.

Below we give a review of results for cyclotomic polynomials of order a most $3$, where we assume $p<q<r$.

For a cyclotomic polynomial of order $1$ obviously $A_p=1$ and $S_p=L_p=Q_p=p$.

For binary cyclotomic polynomials A.Migotti \cite{Migotti-Binary} proved that $A_{pq}=1$. L.Carlitz \cite{Carlitz-Terms} obtained $S_{pq}=Q_{pq}=2p^*q^*-1<\frac12pq$, where $p^*\in\{1,2,\ldots,q-1\}$ is the inverse of $p$ modulo $q$ and $q^*$ is defined similarly. The author \cite{Bzdega-GeneralBeiter} found such $p$ and $q$ that $L_{pq}\ge(\frac4{\pi^2}-\epsilon)pq$, but no upper bound for $L_{pq}$ has been known so far.

For ternary cyclotomic polynomials G.Bachman \cite{Bachman-TernaryBounds} proved that $A_{pqr}\le\frac34p$. There is a conjecture of Y.Gallot and P.Moree \cite{GallotMoree-BeiterCounter} that $A_{pqr}\le\frac23p$ and it is known that the constant $\frac23$ cannot be smaller. The author \cite{Bzdega-Height} proved that $S_{pqr}\le \frac{15}{32}p^2qr$ and if the conjecture $A_{pqr}\le\frac23p$ is true, then $S_{pqr}\le\frac49p^2qr$. Furthermore the author \cite{Bzdega-GeneralBeiter} proved that $L_{pqr}\ge(\frac8{3\pi^3}-\epsilon)p^2qr$. No upper bound for $L_{pqr}$ and almost nothing about $Q_{pqr}$ has been known so far.

\section{Main results}

As we mentioned in the previous section, $M_n$ gives the optimal order for $A_n$, $S_n/n$, $\sqrt{Q(n)/n}$ and $L_n/n$. We define the following constants:
\begin{align*}
\mathcal{B}_k=\limsup_{p_1\to\infty}\frac{A_n}{M_n}, \quad & \quad \mathcal{B}_k^\Sigma=\limsup_{p_1\to\infty}\frac{S_n/n}{M_n}, \\
\mathcal{B}_k^\square=\limsup_{p_1\to\infty}\frac{\sqrt{Q_n/n}}{M_n}, \quad & \quad \mathcal{B}_k^\circ= \limsup_{p_1\to\infty}\frac{L_n/n}{M_n},
\end{align*}
where $\limsup$ is taken over all $n=p_1p_2\ldots p_k$ with $p_1<p_2<\ldots<p_k$. Clearly, $\mathcal{B}_k^\circ\le \mathcal{B}_k^\Sigma\le \mathcal{B}_k^\square \le \mathcal{B}_k$ for $k>0$. We prove the following theorems.

\begin{theorem}\label{t: B2}
For binary cyclotomic polynomials $\mathcal{B}_2^\circ=\frac4{\pi^2}$.
\end{theorem}

\begin{theorem}\label{t: B3}
For ternary cyclotomic polynomials:
\begin{itemize}
\item[(i)] $\mathcal{B}_3^\circ=\frac1{\pi^2}$,
\item[(ii)] $\sqrt{3/2}/\pi^2 \le \mathcal{B}_3^\square\le\sqrt{1/12}$.
\item[(iii)] $\mathcal{B}_3^\Sigma<0.2731$.
\end{itemize}
\end{theorem}

An upper bound on $Q_{pqr}$ is given later by Theorem \ref{t: B3-sqr-formula}.

Summarizing, the known results for cyclotomic polynomials of order $2$ are:
$$\mathcal{B}_2^\circ = \frac{4}{\pi^2}, \quad \mathcal{B}_2^\Sigma=\frac12, \quad \mathcal{B}_2^\square=\frac{\sqrt2}{2}, \quad \mathcal{B}_2=1,$$
and for order $3$ we have:
$$\mathcal{B}_3^\circ=\frac{1}{\pi^2} \le \mathcal{B}_3^\Sigma\le 0.2731, \quad \frac{\sqrt{3/2}}{\pi^2}\le \mathcal{B}_3^\square\le \sqrt{1/12} < \frac23\le\ \mathcal{B}_3\le\frac34.$$

These theorems allow us to improve the bound on $B_k$ from \cite{Bzdega-Height}.

\begin{theorem}\label{t: Bn}
For $n=p_1p_2\ldots p_k$ we have
$$\mathcal{B}_k \le (C+\epsilon_k)^{2^k}M_n,$$
where $C<0.859125$ and $\epsilon_k\to0$ with $k\to\infty$.
\end{theorem}

Additionally we use our methods to estimate the number
$$J_{pqr} = \sum_k|a_{pqr}(k)-a_{pqr}(k-1)|$$
of jumps of ternary cyclotomic coefficients studied by the author \cite{Bzdega-Jumps} and Camburu, Ciolan, Luca, Moree and Shparlinski \cite{Camburu-Gaps}.

\begin{theorem} \label{t: J}
For some constant $c$ we have $J_{pqr} \le cpqrU^2$, where $U=\max\{u_{qr},u_{rp},u_{pq}\}$, $u_{pq}=u_{qp}=\frac1p\min\{q^*,p-q^*\}-\frac1{2pq}$ and $q^*$ is the inverse of $q$ modulo $p$, etc.
\end{theorem}

At the end we prove the following result for the so-called relatives of cyclotomic polynomials, i.e. polynomials of form
$$P_n(z)=\frac{(1-z^n)\prod_{1\le i<j\le k}(1-z^{n/p_ip_j})}{\prod_{i=1}^k(1-z^{n/p_i})},$$
where $n=p_1p_2\ldots p_k$, introduced in \cite{Liu-Relative} by Liu.

\begin{theorem}\label{t: relative}
Let
$$L(k) = \sup_{\omega(n)=k}\max_{|z|=1}\frac{|P_n(z)|}n.$$
Then $\log L(k) \ge \frac{\log2}{2}k^2+O(k\log k)$.
\end{theorem}

In \cite{Liu-Relative} was proved that $\log M(k)=\frac{\log2}{2}k^2+O(k\log k)$, where $M(k)$ is maximal absolute value of a coefficient of $P_n$ with $\omega(n)=k$, so in fact in Theorem \ref{t: relative} we have equality. Theorem \ref{t: relative} gives a constructive and a bit simpler proof of the lower bound $\log M(k)\ge\frac{\log2}{2}k^2+O(k\log k)$.

\section{Preliminaries}

In this section we introduce additional notation used in the paper and we present the main ideas of the proofs.

Let $P(z)=\prod_{d\in D}(1-z^d)^{j_d}$ be a polynomial with integers $j_d$ (not necessarily positive) and $\text{lcm}(D)=n=p_1p_2\ldots p_k$, where $2<p_1<p_2<\ldots<p_k$ are primes. We are interested in the values of $|P(z)|$ for $|z|=1$, so let $z=e^{2\pi i x}$. We have $|1-z^d|=2s(dx)$, where $s(x)=|\sin(\pi x)|$. Therefore
$$F(x):= |P(z)| = 2^{\sum_d j_d}\prod_d s(dx)^{j_d}.$$

We can only consider $x\in[-1/2,1/2)$, because $F(x)=F(x+1)$. Let $x=\frac{N+t}{n}$, where $|N|<n/2$ is an integer and $t\in[-1/2,1/2)$. by the Chinese remainder theorem we can uniquely write $N=N_{a_1,a_2,\ldots,a_k}$, where $N\equiv a_i \pmod{p_i}$ and $|a_i|<p_i/2$ for $i=1,2,\ldots,k$. This substitution will allow as to deal with some expressions of form $s(dx)$ quite easily.

In order to estimate the sum $Q$ of squares of coefficients of $P$ we use the Parseval identity
$$Q = \int_{-1/2}^{1/2}|P(e^{2\pi i x})|^2dx = \int_{-1/2}^{1/2}F(x)^2dx = \frac1n\sum_{a_1,a_2,\ldots,a_k}I_{a_1,a_2,\ldots,a_k},$$
where
$$I_{a_1,a_2,\ldots,a_k} = \int_{-1/2}^{1/2}F\left(\frac{N_{a_1,a_2,\ldots,a_k}+t}{n}\right)^2dt$$
and the sum is over all $|a_i|<p_1/2$, $i=1,2,\ldots,k$.

Throughout the paper we use the notation
$$s(x) = |\sin(\pi x)|, \qquad s_d(x) = s(x/d).$$
We also use the following asymptotic notation: $f(n)\ll g(n)$ if there exists an absolute constant $c$ such that $f(n) < cg(n)$ and $f(n) \lesssim g(n)$ if $f(n) \le (1+o(1))g(n)$.

The following lemmas are crucial in the proofs in the next sections.

\begin{lemma}\label{l: s_p}
Let $n=p_1p_2\ldots p_k$, $N=N_{a_1,a_2,\ldots,a_k}$ and $x=\frac{N+t}{n}$. Then
\begin{itemize}
\item[(i)] $s(nx)=s(t)$,
\item[(ii)] $s((n/p_i)x)=s_{p_i}(a_i+t)$,
\item[(iii)] $s((n/p_ip_j)x)=s_{p_ip_j}((a_j-a_i)p_ip_i^*+a_i+t)=s_{p_ip_j}((a_i-a_j)p_jp_j^*+a_j+t)$,
\end{itemize}
where $p_i^*$ is the inverse of $p_i$ modulo $p_j$ and $p_j^*$ is the inverse of $p_j$ modulo $p_i$.
\end{lemma}

\begin{proof}
Parts (i) and (ii) are trivial. For (iii) note that
$$N\equiv a_ip_jp_j^*+a_jp_ip_i^* \equiv a_i(1-p_ip_i^*)+a_jp_ip_i^* = (a_j-a_i)p_ip_i^*+a_i \pmod{p_ip_j}.$$
\end{proof}

\begin{lemma}\label{l: s_p asymptotic}
Let $N$, $n$ and $x$ be the numbers defined in Lemma \ref{l: s_p}. Let $p_1=\min\{p_1,p_2,\ldots,p_k\}.$ If $p_1\to\infty$, then the following holds.
\begin{itemize}
\item[(i)] $\frac{|a_i+t|}{p_i}\ll s((n/p_i)x) \ll \frac{|a_i+t|}{p_i}$.
\item[(ii)] If $|a_i|<p_i^{1-\epsilon}$ then $s((n/p_i)x)\sim\pi\frac{|a_i+t|}{p_i}$.
\item[(iii)] For $a_i=a_j$ we have $s((n/p_ip_j)x)\ll\frac1{\max\{p_i,p_j\}}$.
\item[(iv)] If $|a_i|,|a_j|<p_1^{1-\epsilon}$ and $a_i\neq a_j$ then $s((n/p_ip_j)x)\sim s_{p_ip_j}((a_j-a_i)p_ip_i^*)$.
\end{itemize}
\end{lemma}

\begin{proof} Parts (i) and (ii) follow easily from the properties of the sine function. To prove (iii), note that by Lemma \ref{l: s_p} we have
$$s((n/p_ip_j)x)=s_{p_ip_j}(a_i+t) \ll \frac{\min\{p_i,p_j\}}{p_ip_j} = \frac1{\max\{p_i,p_j\}}.$$
For the last part note that $s_{p_ip_j}((a_i-a_j)p_jp_j^*) \gg \frac1{p_i}$ and $\frac{a_j+t}{p_ip_j} \ll \frac1{p_i^{1+\epsilon}}$.
\end{proof}

The following fact is easy and we omit its proof.

\begin{proposition}\label{p: s_p quotient}
We have $s(px)/s(x)\le p$ and $s(px)s(qx)/s(x)\le\min\{p,q\}$.
\end{proposition}

Let
$$F_n(x) = |\Phi_n(e^{2\pi ix})|.$$
To work with $\Phi_n$ on the unit circle we use the following formula obtained in \cite{Bzdega-GeneralBeiter}.

\begin{proposition} \label{p: |Phi_n|}
For $n>1$ we have $F_n(x) = \prod_{d\mid n}s(dx)^{\mu(n/d)}$.
\end{proposition}

\section{Warm up: binary polynomials}

\begin{proof}[Proof of Theorem \ref{t: B2}.]

By the results from \cite{Bzdega-GeneralBeiter} we already know that for all $\epsilon>0$ there exist $p$ and $q$ such that $L_{pq}\ge(\frac4{\pi^2}-\epsilon)pq$. Therefore in order to prove Theorem \ref{t: B2} it is enough to show that $L_{pq}\lesssim 4/\pi^2$ with $p=\min\{p,q\}\to\infty$.

Let $2<p<q$ and $F=F_{pq}$. Like explained in the previous section, every $x\in[-1/2,1/2)$ can be uniquely expressed as $x=\frac{N_{a,b}+t}{pq}$ with $|a|<p/2$ and $|b|<q/2$. Let $N=N_{a,b}$.

By Proposition \ref{p: |Phi_n|} and Lemma \ref{l: s_p} we have
$$F(x) = \frac{s(t)s_{pq}(N+t)}{s_p(a+t)s_q(b+t)}.$$

Now we use Lemma \ref{l: s_p asymptotic} and Proposition \ref{p: s_p quotient} to estimate $F(x)$. We consider three cases.

\medskip\noindent\emph{Case 1.} $a=b=0$. Then $N$=0 and $F(x) \ll 1$.

\medskip\noindent\emph{Case 2.} $a,b\neq0$. Then
$$F(x) \le \frac{s(t)}{s_p(a+t)s_q(b+t)} \ll \frac{pq}{|ab|}.$$
If $|a|>p^{1-\epsilon}$ or $|b|>p^{1-\epsilon}$ then $F(x) \ll p^\epsilon q$. Otherwise
$$F(x) \lesssim \frac{pq}{\pi^2}\cdot\frac{s(t)}{|a+t|\cdot|b+t|} \le \frac4{\pi^2}pq.$$

\medskip\noindent\emph{Case 3.} $a\neq0$ and $b=0$ (or reverse, which is analogous). For $|a|>p^{1-\epsilon}$ we have $F(x) \lesssim \frac{pq}{|a|} < p^\epsilon q$. If $|a|\le p^{1-\epsilon}$, then
$$F(x) \lesssim \frac{pq}{\pi^2}\cdot\frac{s(t)}{|t|\cdot|a+t|} \le \frac{pq}{\pi^2}\cdot\frac{s(t)}{|t|\cdot(1-|t|)} \le \frac4{\pi^2}pq$$
by elementary computations.

\medskip Recall that $L_{pq}=\max_x F(x)$, so the proof os done.
\end{proof}

Let us add that the inequality $L_{pq}\le\frac4{\pi^2}pq$ is not true in general. For example, let us fix $p$, choose $q\equiv-2\pmod p$ and $x=\frac{pq-q-1}{2pq}$. Then by using the expansion of $\sin x$ we obtain
$$\frac{F(x)}{pq} \to \frac4{\pi^2}+\frac{2\pi^2-3}{6\pi^2}p^{-2}+O(p^{-4})$$
with $q\to\infty$.

It justifies the assumption $p_1\to\infty$ in the definition of $\mathcal{B}_k^\circ$.

\section{Ternary polynomials: maximum on circle}

In this section we prove part (i) of Theorem \ref{t: B3}. It is an instant consequence of Lemmas \ref{l: B3-circ-upp} and \ref{l: B3-circ-low} below.

Let $2<p<q<r$, $F=F_{pqr}$ and $N=N_{a,b,c}$, $|a|<p/2$, $|b|<q/2$, $|c|<r/2$.

\begin{lemma}\label{l: B3-circ-upp}
Let $x=\frac{N_{a,b,c}+t}{pqr}$.
\begin{itemize}
\item[(i)] If $a=b=c=0$ then $F_{pqr}(x)\ll1$.
\item[(ii)] If $a=b=0$ and $c\neq0$ then
$$F_{pqr}(x) \ll \left\{\begin{array}{ll} 1, & \text{for } |dpq|<r/2, \\ r^2/|d|\ll pqr, & \text{for } |dpq|>r/2, \end{array}\right.$$
where $|d|<r/2$ and $dpq\equiv c \pmod r$. \\
The analogous bound holds for any permutation of $(p,q,r)$ and the appropriate permutation of $(a,b,c)$.
\item[(iii)] For $b=c\neq0$ we have
$$F_{pqr}(x) \ll \left\{\begin{array}{ll} pqr/b^2, & \text{for } a=0, \\ pqr/|ab^2|, & \text{for } a\neq0, \end{array}\right.$$
As in the previous case, this bound has its symmetric versions.
\item[(iv)] For distinct $a,b,c$ we have $F_{pqr}(x)\lesssim\frac1{\pi^2}p^2qr$ with $p=\min\{p,q,r\}\to\infty$.
\end{itemize}
\end{lemma}

\begin{proof}
By Proposition \ref{p: |Phi_n|} and Lemma \ref{l: s_p} we have
$$F_{pqr}(x) = \frac{s(t)s_{qr}(N+t)s_{rp}(N+t)s_{pq}(N+t)}{s_{pqr}(N+t)s_p(a+t)s_q(b+t)s_r(c+t)}.$$
We deal with all parts separately. As for the binary case, we write $F$ instead of $F_{pqr}$. Part (i) is trivial. In the remaining parts we often use Lemmas \ref{l: s_p} and \ref{l: s_p asymptotic} and Proposition \ref{p: s_p quotient}.

\medskip\noindent\emph{Part (ii).} Let $p'$ and $q'$ be the inverses of $p$ and $q$ modulo $r$. Obviously $d\neq0$. Then
$$F(x) = \frac{s(t)s_{pq}(t)}{s_p(t)s_q(t)}\cdot\frac{s_{qr}(cpp'+t)s_{rp}(cpp'+t)}{s_{pqr}(cpqp'q'+t)s_r(c+t)} \ll \frac{s_r(dp)s_r(dq)}{s_r(dpq)s_r(d)}.$$
Now it is clear that if $|dpq|<r/2$ then $F_{pqr}(x) \ll 1$. For all $d$ we have
$$F(x) \ll \frac1{(1/r)(|d|/r)} = r^2/|d|.$$
Similarly we deal with cases symmetric to this one.

\medskip\noindent\emph{Part (iii).} For $a=0$ we have
$$F(x) \ll \frac{pqr\min\{q,r\}}{|bc|\max\{q,r\}} < \frac{pqr}{|bc|}.$$
If $a\neq0$, then we use the bound $s(t)/s_p(a+t)\ll p/|a|$ instead of $s(t)/s_p(0+t)\le p$.

\medskip\noindent\emph{Part (iv).} We have
$$F(x)\le p\cdot\frac{s(t)}{s_p(a+t)s_q(b+t)s_r(c+t)}.$$
Because at most one of $a,b,c$ equals $0$, the quotient above is well defined, as we may replace $t=0$ by $t\to0$ if necessary. If $\max\{|a|,|b|,|c|\}>p^{1-\epsilon}$, then $F(x) \ll p^{2+\epsilon}qr$, so let $|a|,|b|,|c|\le p^{1-\epsilon}$. In this case
$$F(x) \lesssim \frac{p^2qr}{\pi^3}\cdot\frac{s(t)}{|a+t|\cdot|b+t|\cdot|c+t|}.$$
By elementary computations, the last quotient is maximal for $\{a,b,c\}=\{-1,0,1\}$ and $t\to0$. The limit equals $\pi$, which completes the proof.
\end{proof}

The following Lemma gives an explicit example of $p,q,r$ for which $L_{pqr}$ is large.

\begin{lemma}\label{l: B3-circ-low}
Let $q\equiv r\equiv2\pmod p$, $r\equiv\frac{-4}{p-1}\pmod q$, $N=r\cdot\frac{p-1}2+1$ and $x=\frac{N}{pqr}$. Then $F_{pqr}(x)\sim\frac1{\pi^2}p^2qr$ with $p\to\infty$ and $\frac qp\to\infty$.
\end{lemma}

\begin{proof}
We have $N\equiv0\pmod p$, $N\equiv -1\pmod q$ and $N\equiv1\pmod r$, so
$$F(x)\sim\frac{pqr}{\pi^2}\cdot\frac{s_{qr}(N)}{s_{pqr}(N)}\cdot s_{pq}(N)s_{pr}(N).$$
Taking $q/p\to\infty$ we obtain $s_{pr}(N)\to1$ and by Lemma \ref{l: s_p asymptotic}, $s_{pq}(N) \sim s_{pq}(qq^*)\sim1$. Finally
$$s_{qr}(N)/s_{pqr}(N) \sim s(p/2q)/s(1/2q) \sim p$$
with $q/p\to\infty$.
\end{proof}

\section{Ternary polynomials: sum of squares}

In this section we derive an upper bound on $Q_{pqr}$ and use it to prove the second part of Theorem \ref{t: B3}. As mentioned in Preliminaries,
$$Q_{pqr} = \frac{1}{pqr}\sum_{|a|<p/q;\; |b|<q/2;\; |c|<r/2}I_{a,b,c},$$
where
$$I_{a,b,c}=\int_{-1/2}^{1/2}F\left(\frac{N_{a,b,c}+t}{pqr}\right)^2dt.$$
First we deal with some specific triples $(a,b,c)$.

\begin{lemma}\label{l: B3-sqr-specific}
We have
\begin{itemize}
\item[(i)] $I_{0,0,0} \ll 1$,
\item[(ii)] $\sum_{c\neq0}I_{0,0,c} \ll (pqr)^2$, similarly for $I_{0,b,0}$ and $I_{a,0,0}$,
\item[(iii)] $\sum_{b\neq0}I_{a,b,b} \ll (pqr)^2$, similarly for $I_{a,b,a}$ and $I_{a,a,c}$,
\item[(iv)] $\sum_{\max\{a,b,c\}>p^{1-\epsilon}}I_{a,b,c} \ll p^{3+\epsilon}q^2r^2$.
\end{itemize}
\end{lemma}

\begin{proof}
For (i) -- (iii) we use the analogous parts of Lemma \ref{l: B3-circ-upp}. Part (i) is again trivial.

\medskip\noindent\emph{Part (ii).} By Lemma \ref{l: B3-circ-upp}
$$\sum_{c\neq0}I_{0,0,c} \ll \frac{r^2}{p^2q^2}+r^4\sum_{d>r/pq}\frac1{d^2} \ll (pqr)^2,$$
where we used the known fact that $\sum_{k\ge n}\frac1{k^2}\ll\frac1n$.

\medskip\noindent\emph{Part (iii).} Using the bounds from (iii) of Lemma \ref{l: B3-circ-upp} we obtain
$$\sum_{b\neq0}I_{a,b,b}\ll (pqr)^2\sum_{b\neq0}\frac1{b^4}+(pqr)^2\sum_{a,b\neq0}\frac{1}{a^2b^4} \ll (pqr)^2.$$

\medskip\noindent\emph{Part (iv).} By the previous cases, we may consider only distinct $a,b,c$. For such triples we have
$$F(x) \ll \frac{p^2qr\cdot s(t)}{|a+t|\cdot|b+t|\cdot|c+t|}$$
which yields
\begin{align*}
\sum I_{a,b,c} & \ll
p^4q^2r^2\sum_{c>p^{1-\epsilon}}\frac{1}{c^2}\sum_{a,b\neq(0,0)}
\int_{-1/2}^{1/2}\left(\frac{s(t)}{|a+t|\cdot|b+t|}\right)^2dt \\
& \ll p^4q^2r^2\sum_{c>p^{1-\epsilon}}\frac{1}{c^2} \ll p^{3+\epsilon}q^2r^2.
\end{align*}
It completes the proof of the last part.
\end{proof}

Now we are ready to prove the following theorem.

\begin{theorem} \label{t: B3-sqr-formula}
Let $x=\frac1p\min\{q',p-q'\}$ and $y=\frac1p\min\{r',p-r'\}$, where $q'$ and $r'$ are the inverses of $q$ and $r$ modulo $p$. Without loss of generality we assume that $x\le y$. Then
$$\frac{Q_{pqr}}{p^3qr} \lesssim \frac16P(x,y)+\frac1{12}f(x,y),$$
where
\begin{align*}
P(x,y) & = 2x -11x^2 +26x^3 -17x^4 -5y^2 +18y^3-17y^4 \\
& \quad +12xy -24x^2y -12xy^2 +24x^2y^2, \\
f(x,y) & = \{2x+y\}^2(1-\{2x+y\})^2 +\{2x-y\}^2(1-\{2x-y\})^2 \\
& \quad +\{2y+x\}^2(1-\{2y+x\})^2 +\{2y-x\}^2(1-\{2y-x\})^2
\end{align*}
and $\{\cdot\}$ denotes the fractional part of given real number.
\end{theorem}

\begin{proof}
By Lemma \ref{l: B3-sqr-specific} we may focus on distinct $a,b,c\le p^{1-\epsilon}$. By Lemma \ref{l: s_p} and \ref{l: s_p asymptotic},
$$F\left(\frac{N+t}{pqr}\right) \lesssim \frac{p^2qr}{\pi^3}\frac{s(t)s((a-b)x)s((a-c)y)}{|a+t|\cdot|b+t|\cdot|c+t|}.$$
By Lemma \ref{l: B3-sqr-specific} and then by the substitution $m=a-b$, $n=a-c$ and $u=a+t$, we obtain
\begin{align*}
\frac{Q_{pqr}}{p^3qr} & \lesssim \pi^{-6}\sum_{a,b,c}s((a-b)x)^2s((a-c)y)^2\int_{-1/2}^{1/2}\left(\frac{s(t)}{(a+t)(b+t)(c+t)}\right)^2dt \\
& \sim \pi^{-6}\sum_{m,n\neq0;\; m\neq n}
s(mx)^2s(ny)^2\int_{-\infty}^{+\infty}\left(\frac{s(u)}{u(u-m)(u-n)}\right)^2du.
\end{align*}
Computing of the integral is a routine; it equals
$$\pi^2\left(\frac1{m^2n^2}+\frac1{m^2(m-n)^2}+\frac1{n^2(m-n)^2}\right).$$
After reorganizing the variables $m$ and $n$ we arrive at the following asymptotic bound on $Q_{pqr}/p^3qr$:
$$\pi^{-4}\sum_{m,n\neq0;\; m\neq n}\frac{s(mx)^2s(ny)^2 + s(mx)^2s((m+n)y)^2 + s((m+n)x)^2s(ny)^2}{m^2n^2}.$$
We may express this bound as $S_1-S_2$, where $S_1$ runs over all $m,n\neq0$ and $S_2$ runs over $m=n\neq0$, i.e.
$$S_1=\pi^{-4}\sum_{m,n\neq0}\frac{(\ldots)}{m^2n^2}, \quad S_2=\pi^{-4}\sum_{n\neq0}\frac{(\ldots)}{n^4}.$$

Recall that the $k$th Bernoulli polynomial is
$$B_k(x) = \sum_{j=0}^k\binom kj b_{k-j}x^j,$$
where $k>0$, and $b_j$ are the Bernoulli numbers. Particularly
$$B_2(x) = x^2-x+\frac16, \qquad B_4(x) = x^4-2x^3+x^2-\frac1{30}.$$
The Fourier series of $B_k(x)$ is given by
$$B_k(x)=-\frac{k!}{(2\pi i)^k}\sum_{j\neq0}\frac{e^{2\pi ijx}}{j^k}$$
for $0\le x\le 1$ and $k\ge2$.

Let $e(x)=e^{2\pi ix}$. After some straightforward computations we obtain
\begin{align*}
S_1 & = \frac{1}{16\pi^4}\sum_{m,n\neq0}\frac{1}{m^2n^2}
\Big(12-8e(mx)-8e(ny) \\
& \qquad +4e(mx+ny)-4e(my+ny)-4e(mx+nx) \\
& \qquad +2e(m(y+x)+ny)+2e(m(y-x)+ny) \\
& \qquad +2e(mx+n(y+x))+2e(mx+n(y-x))\Big).
\end{align*}
Now we use the equality
$$\sum_{m,n\neq0}\frac{e(mu+nv)}{m^2n^2}=4\pi^4B_2(\{u\})B_2(\{v\}).$$
Notice that $0<x\le y <\frac12$, so we obtain
\begin{align*}
S_1 & = 3B_2(0)^2 -2B_2(x)B_2(0) -2B_2(y)B_2(0) +B_2(x)B_2(y) \\
& \qquad -B_2(x)^2 -B_2(y)^2 +\frac12B_2(y+x)B_2(y) +\frac12B_2(y-x)B_2(y) \\
& \qquad +\frac12B_2(x)B_2(y+x) +\frac12B_2(x)B_2(y-x).
\end{align*}
After some elementary calculations we have
$$S_1=\frac13x +2xy -x^2 +x^3 -2xy^2 -3x^2y +3x^2y^2.$$
It remains to simplify
\begin{align*}
S_2 & = \frac{1}{16\pi^4}\sum_{n\neq0}\frac{1}{n^4}
\Big(12 -8e(nx) -8e(ny) +2e(n(y+x)) +2e(n(y-x)) \\
& \qquad -4e(2nx) -4e(2ny) +2e(n(2y-x)) \\
& \qquad +2e(n(2y+x)) +2e(n(2x-y)) +2e(n(2x+y))\Big).
\end{align*}
We use the equality
$$\sum_{n\neq0}\frac{e(nu)}{n^4} = -\frac23\pi^4B_4(\{u\}).$$
By $0<x\le y<\frac12$ we have
\begin{align*}
S_2 & = -\frac12B_4(0) +\frac13B_4(x) +\frac13B_4(y) +\frac16B_4(2x) +\frac16B_4(2y) \\
& \qquad -\frac1{12}B_4(y+x) -\frac1{12}B_4(y-x) \\
& \qquad -\frac1{12}\left(B_4(2y-x) +B_4(\{2y+x\}) +B_4(\{2x-y\}) +B_4(\{2x+y\})\right).
\end{align*}
By elementary computations
$$S_2 = \frac{17}6x^4 +\frac{17}6y^4 -\frac{10}3x^3 -3y^3 +\frac56x^2 +\frac56y^2-x^2y^2 +x^2y -\frac1{12}f(x,y).$$
Since $\frac{Q_{pqr}}{p^3qr}\lesssim S_1-S_2$, we obtain the assertion of Theorem \ref{t: B3-sqr-formula} by verifying the value of $S_1-S_2$.
\end{proof}

\begin{proof}[Proof of Theorem \ref{t: B3}(ii).]
By Theorem \ref{t: B3-sqr-formula}, for the upper bound it is enough to prove that $f(x,y)\le\frac14$ and $P(x,y)\le\frac38$. The first inequality is obvious. As for the second one, note that
\begin{align*}
\frac{\partial P(x,y)}{\partial x} & = (2+12y) + (-22-48y+48y^2)x + 78x^2 - 68x^3 \\
& \ge 2+12x -34x + 78x^2 - 68x^3 = 2-22x+78x^2-68x^3 \ge 0
\end{align*}
for $0<x\le y<\frac12$. Furthermore
$$\frac{\partial P(y,y)}{\partial y} = 2-8y+24y^2-30y^3\ge0$$
for $0<y<\frac12.$ Thus
$$P(x,y)\le P(y,y) \le P(1/2,1/2)=3/8.$$

In order to prove the lower bound, we again use the polynomial $\Phi_{pqr}$ with $q\equiv r\equiv2\pmod p$ and $r\equiv\frac{-4}{p-1}\pmod q$. We have $N_{a,a-1,a+1}=r\cdot\frac{p-1}2+1+a$ for $|a|<p/2$ and
$$s_{pqr}(N_{a,a-1,a+1}+t)/s_{qr}(N_{a,a-1,a+1}+t)\sim p$$
with $q/p\to\infty$. Moreover,
$$s_{rp}(N_{a,a-1,a+1}+t)\sim1, \quad s_{pq}(N_{a,a-1,a+1}+t)\sim1$$
with $p\to\infty$. Therefore for $|a|<p^{1-\epsilon}$ we have
$$I_{a,a-1,a+1} \sim \frac{p^3qr}{\pi^6}\int_{-1/2}^{1/2}\left(\frac{s(t)}{(a+t)(a-1+t)(a+1+t)}\right)^2dt.$$
Thus
$$\frac{Q_{pqr}}{p^3qr} \gtrsim \pi^{-6}\int_{-\infty}^{+\infty}\left(\frac{s(x)}{x(x-1)(x+1)}\right)^2dx = \frac{3}{2\pi^4},$$
which completes the proof.
\end{proof}

\section{Ternary polynomials: sum of absolute values}

\begin{proof}[Proof of Theorem \ref{t: B3}(iii).]
Let $\Phi_{pqr}(x)=\sum_na(n)x^n$. Put $a=\frac{S_{pqr}}{p^2qr}$, $m=\frac{A_{pqr}}{p}$ and $|a(n)|/p = a+r_n$. Clearly, $\sum_nr_n=0$. Then
$$\frac1{12}\gtrsim\frac{Q_{pqr}}{p^3qr}=\frac1{pqr}\sum(a+r_n)^2=a^2+\frac1{pqr}\sum_nr_n^2.$$
We will estimate the sum $R=\sum_nr_n^2$ from below. Note that $|a(n)-a(n-qr)|\le2$, because
$$(1-x^{qr})\Phi_{pqr}(x) = (1-x^r)\underbrace{(1+x+\ldots+x^{q-1})\Phi_{qr}(x^p)}_{\text{flat}}.$$
The underbraced polynomial is flat because every two nonzero consecutive coefficients of $\Phi_{qr}$ are $\pm1$ and $\mp1$. Therefore we may consider a continuous function $f:[0,1]\to[0,m]$ satisfying $a+r_n=f(n/pqr)$. Moreover, $f(0)=f(1)=0$ and $|f(x)-f(y)|\le2|x-y|$ for $x,y\in[0,1]$. For $p\to\infty$ we have
$$\int_0^1f(x)dx\sim a, \quad \int_0^1(f(x)-a)^2dx\sim R.$$
The value of $R$ is minimal for the function
$$f(x)=\left\{\begin{array}{ll}
2x, & \text{for } 0 \le x \le m/2, \\
m, & \text{for } m/2 \le x \le 1-m/2, \\
2(1-x) & \text{for } 1-m/2 \le x \le 1,
\end{array}\right.$$
where the optimal $m$ equals $1-\sqrt{1-2a}$ (note that we already know that $a<1/2$). So we have
$$\int_0^1(f(x)-a)^2dx = -\frac23m^3+(m-a)^2+am^2.$$
We need to solve the system
$$\left\{\begin{array}{ll}
m & = 1-\sqrt{1-2a}, \\
\frac1{12} & \ge a^2-\frac23m^3+(m-a)^2+am^2.
\end{array}\right.$$
By numerical computations, the solution is
$$a\le 0.273099\ldots < 0.2731$$
and the proof is done.
\end{proof}

\section{General case}

Let us recall the following lemma from \cite{Bzdega-Height}.

\begin{lemma}\label{l: Bn-rec}
Let $p_1<p_2<\ldots<p_k$ be primes and $n=p_1p_2\ldots p_k$. Then
$$\Phi_n(x)=f_n(x) \cdot \prod_{j=1}^{k-2}\prod_{i=j+2}^k\Phi_{p_1\ldots p_j}(x^{p_{j+2}\ldots p_k/p_i}),$$
where $f_n$ is a formal power series satisfying $f_n(x) = (1-x^n) \cdot \frac{\prod_{i=2}^k(1-x^{n/p_1p_i})}{\prod_{i=1}^k(1-x^{n/p_i})}$.
\end{lemma}

Let $f_n^*$ be the polynomial of degree smaller than $n$, satisfying $f_n^*(x)\equiv f_n(x)\pmod{x^n}$. In the same paper the author proved that for $k\ge2$ the height of $f_n^*$ does not exceed $\binom{k-2}{\lfloor k/2\rfloor-1}$. Here we prove the following bound.

\begin{lemma}\label{l: fn-sum}
The sum of absolute values of coefficients of $f_n^*$ is $\lesssim\frac{2^{k-1}n}{k!}$ with $p_1=\min\{p_1,p_2,\ldots,p_k\}\to\infty$.
\end{lemma}

\begin{proof}
We have
$$f_n^*(x)\equiv \prod_{i=2}^k(1-x^{n/p_1p_i})\sum_{\alpha_1,\ldots,\alpha_k\ge0}x^{\alpha_1n/p_1+\ldots+\alpha_kn/p_k} \pmod{x^n}.$$
The product has the sum of absolute values of coefficients not greater than $2^{k-1}$. Let $p_1\to\infty$. Then the sum of coefficients of $\sum x^{\alpha_1n/p_1+\ldots}$ with exponents smaller than $n$ equals asymptotically the volume of simplex on vertices
$$(p_1,0,\ldots,0), \quad (0,p_2,0,\ldots,0), \quad \ldots, \quad (0,\ldots,0,p_k).$$
It equals $\frac n{k!}$.
\end{proof}

\begin{proof}[Proof of Theorem \ref{t: Bn}]
By Lemma \ref{l: Bn-rec} we have
$$S_{p_1\ldots p_k} \le \frac{2^{k-1}p_1\ldots p_k}{k!}\prod_{j=1}^{k-2}S_{p_1\ldots p_j}^{k-j-1}.$$
By the inequality $S_n \lesssim nM_n \mathcal{B}_{\omega(n)}$ with $\min\{p:p\mid n\}\to\infty$, after some calculations we obtain
$$\mathcal{B}_k^\Sigma \le \frac{2^{k-1}}{k!}\prod_{j=1}^{k-2}(\mathcal{B}_j^\Sigma)^{k-j-1}$$
for $k\ge3$. Put $b_k=\mathcal{B}_k^\Sigma$ for $k=1,2,3$ and $b_k = \frac{2^{k-1}}{k!}\prod_{j=1}^{k-2}b_j^{k-j-1}$ for $k>3$. Then clearly $\mathcal{B}_k^\Sigma\le b_k$. For $k\ge6$ we have
$$\frac{b_k/b_{k-1}}{b_{k-1}/b_{k-2}} = \frac{k-1}{k}b_{k-2},$$
so $b_k=\frac{k-1}{k}b_{k-1}^2$. Furthermore,
$$C:= \lim_{k\to\infty}b_k^{2^{-k}} = b_5^{1/32}\cdot\prod_{k=6}^\infty\left(\frac{k-1}{k}\right)^{2^{-k}},$$
where $b_5 = \frac{2^4}{5!}\cdot b_3\cdot b_2^2\cdot b_1^3 = b_3/30 = \mathcal{B}_3^\Sigma/30$. Based on Theorem \ref{t: B3} and some numerical computations we conclude that $C<0.859125$.

Now by Lemma \ref{l: Bn-rec} and the bound $\binom{k-2}{\lfloor k/2\rfloor-1}$ on the height of $f_n^*$ we have
$$\mathcal{B}_k \le \binom{k-2}{\lfloor k/2\rfloor-1}\prod_{j=1}^{k-2}(\mathcal{B}_j^\Sigma)^{k-j-1} < 2^{k-1}\prod_{j=1}^{k-2}b_j^{k-j-1} = k!b_k,$$
so by verifying that $\lim_{k\to\infty}(k!)^{2^{-k}}=1$ we complete the proof.
\end{proof}

\section{Jumps of ternary cyclotomic coefficients}

\begin{proof}[Proof of Theorem \ref{t: J}.]
In the proof we do not assume that $p<q<r$. Let $F(x)=|(1-z)\Phi_{pqr}(z)|$, where $z=e^{2\pi ix}$. We have
$$J = \frac1{pqr}\sum_{|a|<p/2;\; |b|<q/2;\; |c|<r/2}I_{a,b,c},$$
where $I_{a,b,c} = \int_{-1/2}^{1/2}F\left(\frac{N_{a,b,c}+t}{pqr}\right)^2dt$. Now we need to consider some cases. We omit details, as the computations are similar to those in the proofs of Lemma \ref{l: B3-sqr-specific} and Theorem \ref{t: B3-sqr-formula}.

\medskip\noindent\emph{Case 1.} $a=b=c=0$. We have $I_{0,0,0}\ll (pqr)^{-2}$.

\medskip\noindent\emph{Case 2.} $b=c=0$ and $a\neq0$. Then $F(x)\ll p/a$, so $\sum_{a\neq0}I_{a,0,0} \ll p^2$. Similarly we deal with two symmetric cases.

\medskip\noindent\emph{Case 3.} $b=c\neq0$. Then $f(x)\ll\frac{p}{ab}$ and $\sum_{a,b\neq0}I_{a,b,b} \ll p^2$

\medskip\noindent\emph{Case 4.} $a,b\neq0$, $a\neq b$ and $c=0$. Then
$$f(x) \ll \frac{pqr}{|ab|}s((a-b)u_{pq})s(au_{pr})s(bu_{qr}).$$
Now, depending on $a$ and $b$, we use different bounds on $s(au)$, $s(bu)$ and $s((a-b)u)$. The number $n$ we determine later. For $|a|\le n$ we use $s(au)\ll|a|u$ and for $a>n$ we use $s(au)\le1$. Similarly for $b$ and $a-b$. We have
\begin{align*}
\sum_{0<|a|,|b|\le n}I_{a,b,0} & \ll (pqr)^2U^6\sum_{0<|a|,|b|\le n}(a-b)^2 \ll (pqr)^2U^6n^4, \\
\sum_{0<|a|\le n <|b|}I_{a,b,0} & \ll (pqr)^2U^2\sum_{0<|a|\le n <|b|}\frac1{b^2} \ll (pqr)^2U^2, \\
\sum_{|a|,|b|>n}I_{a,b,0} & \ll (pqr)^2\sum_{|a|,|b|>n}\frac1{a^2b^2} \ll \frac{(pqr)^2}{n^2}.
\end{align*}
The case of the sum $\sum_{0<|b|\le n <|a|}I_{a,b,0}$ is analogous to the second sum above. The optimal choice of $n$ is $1/U$. Then we have
$$\sum_{|a|<p/2;\;|b|<q/2}I_{a,b,0} \ll (pqr)^2U^2.$$

\medskip\noindent\emph{Case 5.} It remains to consider distinct $a,b,c\neq0$. We have
$$f(x) \ll \frac{pqr}{|abc|}s((a-b)u_{pq})s((b-c)u_{qr})s((c-a)u_{rp}).$$
Each case is symmetric to one of the following.
\begin{align*}
\sum_{0<|a|,|b|,|c|\le n}I_{a,b,c} & \ll (pqr)^2U^6\sum_{0<|a|,|b|,|c|\le n}\frac{(a-b)^2(b-c)^2(c-a)^2}{a^2b^2c^2} \\
& \ll (pqr)^2U^6n^4, \\
\sum_{0<|a|,|b|\le n <|c|}I_{a,b,c} & \ll (pqr)^2U^2\sum_{0<|a|,|b|\le n <|c|}\frac{(a-b)^2}{a^2b^2c^2} \ll (pqr)^2U^2, \\
\sum_{|a|,|b|>n;\;|c|>0}I_{a,b,c} & \ll (pqr)^2\sum_{|a|,|b|>n;\;|c|>0}\frac1{a^2b^2c^2} \ll \frac{(pqr)^2}{n^2}.
\end{align*}
Note that the obtained bounds are the same as in the previous case, so again we have
$$\sum_{a,b,c\neq0}I_{a,b,c} \ll (pqr)^2U^2$$
as desired.

\medskip \medskip To complete the proof, note that $\max\{p,q,r\}^2 = o(1)(pqr)^2U^2$.
\end{proof}

\section{Relatives of cyclotomic polynomials}

\begin{proof}[Proof of Theorem \ref{t: relative}]
We consider primes $p_1<p_2<\ldots<p_k$ with $p_1\to\infty$, satisfying the congruences
$$p_j\equiv2(j-i)\pmod{p_i} \qquad\text{for all }1\le i < j \le k.$$
The existence of such primes is guaranteed by the Chinese remainder theorem and Dirichlet's theorem on primes in arithmetic progressions. Put $N=N_{1,2,\ldots,k}$ and $x=\frac{N-\frac12}n$. By Lemma \ref{l: s_p} we have
$$F(z) = |P_n(x)| = 2^{\binom k2 -k+1} \frac{s(1/2)\prod_{1\le i<j\le k}s_{p_ip_j}(N-1/2)}{\prod_{i=1}^ks_{p_i}(i-1/2)}.$$
By using Lemma \ref{l: s_p asymptotic} we obtain $s_{p_i}(i-1/2) \sim \frac{\pi}2\frac{2i-1}{p_i}$ and
$$s_{p_ip_j}(N-1/2) \sim s_{p_ip_j}((i-j)p_jp_j^*)=s_{p_ip_j}(p_j(p_i-1)/2)=1.$$
It gives
$$F(x)\sim \frac{2^{\binom k2 + 1}n}{\pi^k(2k-1)!!} = n\cdot2^{k^2/2+O(k\log k)},$$
which completes the proof.
\end{proof}

\section*{Acknowledgements}

The author is partially supported by the NCN grant no. 2012/07/D/ST1/02111.

\end{document}